\documentclass[11pt]{amsart}
\usepackage{amsmath,amssymb,amsthm}
\usepackage{amsfonts}
\usepackage{mathrsfs}
\usepackage{mathtools}
\usepackage{multirow}
\usepackage{hyperref}
\usepackage[dvips]{graphicx}
\usepackage{color}              
\usepackage{epsfig}
\usepackage{enumerate}
\usepackage[all,cmtip]{xy}
\usepackage{float,caption}

\usepackage{tikz}
\usetikzlibrary{cd, matrix, arrows, calc, intersections, decorations.pathmorphing,backgrounds,positioning,fit, through}
\usetikzlibrary{arrows.meta}
\usetikzlibrary{calc,through,backgrounds, patterns}
\usetikzlibrary{arrows,decorations.pathmorphing, decorations.pathreplacing,backgrounds,positioning,fit}

\usepackage{pgfplots}
\usepgfplotslibrary{polar}
\pgfplotsset{compat=newest}

\usepackage[margin=1.3in]{geometry}
\newtheorem{theorem}{Theorem}[section]
\newtheorem{proposition}[theorem]{Proposition}
\newtheorem{corollary}[theorem]{Corollary}
\newtheorem{lemma}[theorem]{Lemma}

\newtheorem{introthm}{Theorem}
\newtheorem{introcor}[introthm]{Corollary}

\theoremstyle{definition}
\newtheorem{definition}[theorem]{Definition}

\newtheorem*{example*}{Example}

\newtheorem*{claim*}{Claim}

\newtheorem{remark}[theorem]{Remark}
\newtheorem*{remark*}{Remark}

\newcommand{\calE}{{\mathcal E}}

\newcommand{\calG}{{\mathcal G}}
\newcommand{\calH}{{\mathcal H}}

\newcommand{\fG}{\mathfrak{G}}

\newcommand{\A}{{\mathbb A}}

\newcommand{\N}{{\mathbb N}}

\newcommand{\Z}{{\mathbb Z}}

\newcommand{\from}{\colon \thinspace}

\newcommand{\param}{{\mathchoice{\mkern1mu\mbox{\raise2.2pt\hbox{$
\centerdot$}}
\mkern1mu}{\mkern1mu\mbox{\raise2.2pt\hbox{$\centerdot$}}\mkern1mu}{
\mkern1.5mu\centerdot\mkern1.5mu}{\mkern1.5mu\centerdot\mkern1.5mu}}}

\newcommand{\AG}{{\A(\Gamma)}}
\newcommand{\ga}{{\gamma}}

\begin{document}

\title[Excursions of generic geodesics in right-angled Artin groups]{Excursions of generic geodesics in right-angled Artin groups and graph products}


  \author[Yulan Qing]{Yulan Qing}
  \address{Dept. of Mathematics\\ 
  University of Toronto\\
  Toronto, Ontario, Canada M5S 2E4}
  \email{\tt yulan.qing@utoronto.ca}

  \author[Giulio Tiozzo]{Giulio Tiozzo}
  \address{Dept. of Mathematics\\ 
  University of Toronto\\
  Toronto, Ontario, Canada M5S 2E4}
  \email{\tt tiozzo@math.utoronto.ca}

 \date{\today}

\maketitle

\begin{abstract}
Motivated by the notion of \emph{cusp excursion} in geometrically finite hyperbolic manifolds, 
we define a notion of \emph{excursion} in any subgroup of a given group, and study its asymptotic distribution for right-angled Artin groups and graph 
products. 

In particular, for any irreducible right-angled Artin group we show that with respect to the counting measure, the maximal excursion of a generic geodesic in any flat tends to $\log n$, where $n$ is the length of the geodesic. 

In this regard, irreducible RAAGs behave like a free product of groups. In fact, we show that the asymptotic distribution of excursions 
detects the growth rate of the RAAG and whether it is reducible.


\end{abstract}

\section{introduction}

On a finitely generated group $G$, there are fundamentally two ways to take averages over $G$ in 
order to study the asymptotic distribution of group elements:

\begin{enumerate}
\item a finite generating set $S$ defines a \emph{word metric} $\Vert \cdot \Vert$ on $G$, and one can take for each $n$ 
the \emph{counting measure} 
$$P_n := \frac{1}{|S_n|} \sum_{g \in S_n} \delta_g$$
that is, the uniform measure on the sphere $S_n$ of radius $n$. 
The boundary measure obtained as a limit of such measures $(P_n)$ (if it exists, and in an appropriate compactification) is known as a \emph{Patterson-Sullivan measure}; 

\item a measure $\mu$ on $G$ defines a \emph{random walk}, and the limit of the sequence $(\mu^n)$ of convolution measures defines a \emph{harmonic measure} on 
the boundary of $G$. 
\end{enumerate}

If $G$ is non-amenable, different averaging procedures may produce vastly different results, and many authors have considered the question 
of whether these two measures are singular, or if they are in the same measure class. 

One approach to this question is through the \emph{excursion} of generic geodesics.
In particular, if $G < \textup{Isom}(\mathbb{H}^n)$ is a nonuniform lattice, 
one can look at the \emph{excursion} of generic geodesics in the \emph{cusp} of the quotient manifold $\mathbb{H}^n/G$ (\cite{GL}, \cite{GMT1}). 
This technique also works for the mapping class group acting on Teichm\"uller space \cite{GMT2}. 

In group-theoretic terms, the excursion of a geodesic into the cusp may be rephrased as follows. For instance, let $G$ be a torsion-free
non-uniform lattice in $PSL(2, \mathbb{C}) = \textup{Isom}^+(\mathbb{H}^3)$; then the parabolic subgroup $P$ is isomorphic to $\mathbb{Z}^2$, 
and the group is hyperbolic relative to $P$. Then the excursion of the geodesic into the cusp corresponds to the maximal distance traveled 
in any left coset of the parabolic subgroup. 

In this paper, we will define the notion of excursion for general groups, and determine the distribution of excursions for right-angled Artin groups. 

\begin{definition} \label{D:exc1}
Let $G$ be a finitely generated group, and $H < G$ be a subgroup. For any $g \in G$, let us define the \emph{excursion} of $g$ 
in $H$ as follows. 
If $\gamma$ is a geodesic segment in the Cayley graph of $G$, we define the excursion $\mathcal{E}_H(\gamma)$ of $\gamma$ into $H$ as the maximal distance travelled 
by $\gamma$ in a left coset of $H$, i.e.: 
$$\mathcal{E}_H(\gamma) := \max_{t \in G} \textup{diam }(\gamma \cap t H).$$
Then, for a given element $g \in G$ let us define the excursion $\mathcal{E}_H(g)$ as the maximum value of $\mathcal{E}_H(\gamma)$ among 
all geodesic segments $\gamma$ which join $1$ and $g$. 
\end{definition}


\smallskip
\textbf{Example.} To illustrate the definition of excursion, let us consider the group $G = \mathbb{Z} \star \mathbb{Z}^2 = \langle a, b, c \ | \ [b,c] = 1\rangle$. Its Cayley graph for the standard generating set can be seen as a ``tree of planes", 
where each plane is a translate of $H = \mathbb{Z}^2$. 
We will now define the excursion function of a geodesic as the maximal distance traveled by a geodesic 
in a given plane. To make things precise, pick $g \in G$, and suppose that a geodesic representative $\gamma$ of $g$ can be written as 
$$g = a^{\epsilon_1} b^{m_1} c^{n_1} \dots a^{\epsilon_k} b^{m_k} c^{n_k}$$
where $\epsilon_i, m_i, n_i \in \mathbb{Z}$.
For each $i \leq k$, the distance traveled in the $i^{th}$ plane (for the graph metric) is 
$$d_i = |m_i| + |n_i|.$$
Then one sets $\mathcal{E}_H(g) =  \mathcal{E}_H(\gamma) = \max \{ d_1, \dots, d_k \}$.\footnote{Note that one can replace the previous definition of $d_i$ with the euclidean distance $d^{(2)}_i := \sqrt{m_i^2 + n_i^2}$, and one 
has $d_i \sim d^{(2)}_i$ up to uniform multiplicative constants, hence it does not make a difference which definition we take.}

\smallskip

The above example is a free product, but in this paper we will consider more generally excursions in \emph{right-angled Artin groups} and study 
their asymptotic distribution with respect to the counting (or Patterson-Sullivan) measure. 

\smallskip
Let $\Gamma$ be a finite graph, and let $G = \mathbb{A}(\Gamma)$ be its associated right-angled Artin group. We will consider the 
standard generating set, where generators correspond to vertices of $\Gamma$. The Cayley graph $X$ of $\mathbb{A}(\Gamma)$ 
is not in general hyperbolic, and in fact for any complete subgraph $\Gamma' \subseteq \Gamma$ the vertices of $\Gamma'$ generate 
an abelian subgroup of $G$, and its translates in $X$ are flats of rank $k = |\Gamma'|$. 

We will show that the maximal excursion of a geodesic of length $n$ in any family of flats is of the order $\log n$ for a set of geodesics 
of probability which tends to $1$. 

\begin{introthm} \label{T:main}
Let $\A(\Gamma)$ be a non-elementary, irreducible right-angled Artin group, let $\Gamma' \subseteq \Gamma$ be a complete subgraph with $k$ vertices, and let $H = \mathbb{A}(\Gamma') < G$ be the corresponding abelian subgroup. Then there exist $c_1, c_2 > 0$ such that the counting measure satisfies
\begin{equation} \label{E:main-thm}
P_n\left( g \in G \ : \ c_1 \leq \frac{\mathcal{E}_H(g)}{\log n} \leq c_2 \right) \to 1
\end{equation}
as $n \to \infty$.
In fact, the above inequality holds for $c_2 = \frac{k}{\log \lambda}$ where $\lambda$ is the growth rate of $\A(\Gamma)$, and $c_1 = \frac{1}{\log \lambda}-\epsilon$ 
for any $\epsilon > 0$. 
\end{introthm}

Let us recall that a RAAG is \emph{irreducible} if it is not a direct product of smaller RAAGs (equivalently, its opposite graph $\Gamma^{op}$ is connected), 
and \emph{non-elementary} if the graph $\Gamma$ has at least two vertices. 
Note that for right-angled Artin groups the excursion also has a more geometric interpretation in terms of the $CAT(0)$ geometry; 
namely, 
the excursion of $g$ in a vertex group $H = \langle v \rangle$ is the 
maximal number of consecutive $v$-hyperplanes separating $1$ and $g$
(see Section \ref{S:geom}).

\begin{introcor}
The growth rate of the group can be detected from the asymptotics of the excursion: in fact, if $H = \langle v \rangle$ 
is a vertex group and $g_n$ is a group element of length $n$ chosen uniformly at random, then 
\begin{equation} \label{E:detect}
\frac{\mathcal{E}_H(g_n)}{\log n} \to \frac{1}{\log \lambda}
\end{equation}
in probability. 
\end{introcor}

Moreover, it turns out that the distribution of excursions can be also used to characterize irreducible RAAGs. Indeed, we say that the excursion in a given subgroup $H$ is \emph{asymptotically logarithmic} if there exist constant $c_1, c_2 >0$ such that \eqref{E:main-thm} holds. We prove the following (see Theorem \ref{T:all-cases}): 

\begin{introthm} \label{T:irred}
A non-elementary right-angled Artin group $\A(\Gamma)$ is irreducible if and only if it has pure exponential growth and the generic excursion in \emph{every} vertex group of $\Gamma$ is asymptotically logarithmic.
\end{introthm}

For the random walk measure, in the case of relatively hyperbolic $G$, Sisto and Taylor \cite{ST} proved that the generic excursion is logarithmic (see also \cite{Si}).  
Their result also applies to the mapping class group and $\textup{Out}(F_n)$, and it seems likely that their techniques also apply to RAAGs.

In the probabilistic setting, logarithmic distributions for excursions date back a long time: e.g. it has been known at least since R\'enyi \cite{Re} 
that the longest streak of consecutive heads in $n$ coin tosses has size $\log n$ (see e.g. \cite{Sch}); in our language, this is basically equivalent to saying 
that the generic excursion for the simple random walk on $\mathbb{Z} \star \mathbb{Z}$ is logarithmic. 
Moreover, a logarithmic law for the excursion of geodesics in the cusp of a hyperbolic manifold has been famously proved by Sullivan \cite{Su}. 

The proof is based on the study of the generating function (or \emph{spherical growth series}) for the group, and in particular we do not use any random walk 
technique here.
First, we give three definitions of excursion, which we prove to be equivalent for RAAGs: the \emph{algebraic excursion} as given in Definition \ref{D:exc1}, the \emph{geometric excursion} in terms of hyperplanes, and the excursion using a normal form; 
then, in Section \ref{S:growth-series} we construct generating functions for the cardinalities of the set of elements in a sphere of radius $n$ with a given upper bound $E$ on the excursion (Proposition \ref{P:formula}). 
Then, in Section \ref{S:proof} we study how the smallest pole of such a function varies with $E$. Since the group has exact exponential growth, this implies 
that the smallest pole is simple, and the convergence of the power series to the original one is exponential; this establishes the bound on the excursion. 

\medskip
\textbf{Acknowledgements.} We thank Andrew McCormack for providing interesting numerical simulations about the distribution of generic excursions.
We also thank Ilya Gekhtman and Sam Taylor for useful conversations. The second named author is partially supported by NSERC and the Alfred P. Sloan 
Foundation.


\section{background} \label{S:back}

\subsection{Graph products and right-angled Artin groups}
A \emph{graph product of groups} is a construction of larger groups from smaller groups based on a simplicial graph $\Gamma$. Let $\Gamma$ be a simplicial graph, and let us associate to each vertex $v$ a group $G_v$, which we call a \emph{vertex group}. Then the graph product of the $G_v$'s with respect to $\Gamma$ is defined as the quotient $F/R$, where $F$ is the free product of all the $G_v$'s and $R$ is the normal subgroup generated by subgroups of the form $[G_u,G_v]$ whenever there is an edge joining $u$ and $v$ in $\Gamma$. We denote the graph product of groups on a finite graph $\Gamma$ as $\fG \Gamma$.

A \emph{right-angled Artin group (RAAG)} is a specific type of graph product where vertex groups are all isomorphic to $\Z$. 
In fact there is a standard presentation of a right-angled Artin group with respect to a  simplicial graph $\Gamma$ as follows:
\[ \A(\Gamma) := \big \langle v \text{ is a vertex in }\Gamma \ | \ [v, w] = 1, (v, w) \text{ is an edge in }\Gamma \big \rangle. \]
That is to say, generators commute if and only if there is an edge between the two corresponding vertices in $\Gamma$. 
The \emph{opposite graph} $\Gamma^{\rm op}$ is the graph with the same vertex set as $\Gamma$ and whose edge set is the complement of the edge set of $\Gamma$. A right-angled Artin group is called \emph{irreducible} if the opposite graph $\Gamma^{\rm op}$ is connected; that is, the group is not a direct 
product of two RAAGs defined by smaller subgraphs. A right-angled Artin group is \emph{non-elementary} if the defining graph $\Gamma$ has at least $2$ vertices.

\subsection{Growth of groups}
Suppose $G$ is a finitely generated group and $S$ a finite symmetric set of generators. 
Let $Cay(G, S)$ be the Cayley graph of $G$ with respect to $S$. 
Recall the \emph{word length} of an element $G$ is 
$$\Vert g \Vert := \min \{ k \geq 0 \ : \ g = s_1 s_2 \ldots s_k, s_i \in S \}$$ 
and this defines the \emph{word metric} (or \emph{edge metric}) on $Cay(G, S)$ by setting $d(g, h) := \Vert g^{-1} h \Vert$.
Let us now consider the ball of radius $n$ in the Cayley graph:
\[ 
B_{n}(G,S)=\{x\in G\mid x=a_{1}\cdot a_{2}\cdots a_{k}{\mbox{ where }}a_{i}\in S{\mbox{ and }}k\leq n\}
\]
while the \emph{sphere of radius }$n$ is the set of elements of length $n$: 
$$S_n = S_{n}(G,S):=\{ g \in G \ : \ \Vert g \Vert = n \}.$$
To emphasize the group or subgroup in which we calculate word length, we also use $\ell_{H} (g)$ to denote the length of $g$ with respect to the 
word metric in a subgroup $H$. 
Let us define the \emph{growth rate} of $G$ with respect to the generating set $S$ as 
$$\lambda(G) := \lim_{n \to \infty} \left( \#B_n(G, S) \right)^{1/n}.$$
We say a group has \emph{exponential growth} if $\lambda = \lambda(G) > 1$. 
Now, let us denote as $\calG(n): =  \# S_{n}(G,S)$ the cardinality of the set of group elements of length $n$. 
Finally, we say that $G$ has \emph{pure exponential growth} if there exists a constant $C > 0$ such that 
$$\frac{1}{C} \leq \frac{ \calG(n) }{\lambda^n} \leq C \qquad \textup{for any }n,$$
and it has \emph{exact exponential growth} if 
$$\lim_{n \to \infty} \frac{ \calG(n) }{\lambda^n}$$
exists. 
Irreducible, non-elementary right-angled Artin groups have exact exponential growth:

\begin{theorem}[\cite{GTT}, Theorem 11.1] \label{T:exact}
If $G = \A(\Gamma)$ is an irreducible, non-elementary right-angled Artin group with the standard generating set, then $G$ has exact exponential growth. In fact, there exists $\lambda = \lambda(G) > 1$ and $c > 0$ such that 
$$\lim_{n \to \infty} \frac{\calG(n)}{\lambda^n}  = c.$$
\end{theorem}

As a corollary of the previous theorem, one has also a precise characterization of the growth rate for reducible RAAGs. Namely, if $\Gamma^{op}$ is disconnected, 
let us denote as $\Gamma_1^{op}, \dots, \Gamma_r^{op}$ the connected components of $\Gamma^{op}$. Then one has 
$\mathbb{A}(\Gamma) = \mathbb{A}(\Gamma_1) \times \dots \times \mathbb{A}(\Gamma_r)$, hence if 
$\mathbb{A}(\Gamma)$ 
is not abelian it 
has exponential growth $\lambda > 1$ and
$$\calG(n) \sim n^{s-1} \lambda^n$$
where $s$ is the number of irreducible components $\mathbb{A}(\Gamma_i)$ of maximal growth rate. In particular, a RAAG has pure exponential growth 
if and only if is not abelian and it does not have two irreducible components of equal growth.

\begin{remark}
Note that a different, coarser definition of growth is also used often in the literature.
In particular, one may define two nondecreasing positive functions $f$ and $g$ to be equivalent $f \sim g$ if there is a constant $C$ such that
\begin{equation}\label{equivalence}
 f(n/C) \leq g(n) \leq C f(n).
\end{equation}
Then the (coarse) growth rate of the group $G$ can be defined as the corresponding equivalence class of the function $f(n) = \# B_{n}(G,S)$.
This definition is robust with respect to changing generators, but provides a much coarser classification. For instance, 
it does not distinguish between groups of exponential growth and of pure exponential growth, and neither it distinguishes between groups 
of exponential growth with different growth rates. As an example, it is known that the mapping class group has exponential growth, but 
it is not known whether it has pure exponential growth for, say, the standard Humphries generating set, and neither what the growth rate is. 
The definition we use in the present paper, on the contrary, is sensitive to the generating set, and provides much finer information, distinguishing e.g. 
between groups with growth rates, say, $2^n$, $n 2^n$, and $3^n$. 

Finally, let us remark that in this paper we focus on counting elements in the \emph{sphere} $S_n(G, S)$ rather than in the 
\emph{ball} $B_n(G, S)$. It is a simple observation that asymptotic results for counting in spheres are a priori stronger than counting in balls, 
as one gets the counting in balls by averaging over spheres.
\end{remark}

\subsection{Spherical growth series}

The growth series of a group encodes the number of elements of given length in a group as coefficients of a power series. 
If $\Sigma \subseteq G$ is a subset of a group, given a word metric $\Vert \cdot \Vert$ on $G$ we define the \emph{spherical growth series} as the generating function
$$\Sigma(t) := \sum_{g \in \Sigma} t^{\Vert g \Vert}.$$
In particular, if $\Sigma = G$ is the entire group, we can encode the growth of $G$ as coefficients in a power series,  and get 
\[
G(t) = \sum_{n = 0}^\infty \calG(n) t^{n}.
\] 
One can derive the generating function of direct and free products of groups from the generating functions of the respective groups as follows:
\begin{equation} \label{E:direct-prod}
\big ( G_{1} \times G_{2} \big) (t) = G_{1}(t) G_{2}(t)
\end{equation}

\begin{equation} \label{E:free-prod}
\frac{1}{\big ( G_{1} \star G_{2} \big) (t)}= \frac{1}{G_{1}(t)} + \frac{1}{G_{2}(t)} -1.
\end{equation}

\medskip
\textbf{Example.}  As an example, which will guide us through the rest of the argument, one immediately gets by looking at balls in $\mathbb{Z}$:
$$\mathbb{Z}(t) = 1 + 2 \sum_{k = 1}^\infty t^k = \frac{1+t}{1-t}$$
from which and \eqref{E:direct-prod} one gets 
$$\mathbb{Z}^2(t) = \left(\frac{1+t}{1-t}\right)^2.$$ 
Hence, from \eqref{E:free-prod}
\begin{equation} \label{E:Z2}
(\mathbb{Z}\star \mathbb{Z}^2)(t) = \frac{1}{\frac{1}{\mathbb{Z}(t)} + \frac{1}{\mathbb{Z}^2(t)} -1} = 
\frac{1}{ \frac{1-t}{1+t} + \left( \frac{1-t}{1+t} \right)^2 - 1  } =
\frac{(1+t)^2}{1-4t-t^2}
\end{equation}
so the convergence radius of $f_{\mathbb{Z}\star \mathbb{Z}^2}(t)$ equals $\sqrt{5}-2$, the modulus of the smallest root of $1 - 4t - t^2 = 0$.
Thus, the growth rate of $\mathbb{Z}\star \mathbb{Z}^2$ equals
$$\lambda(\mathbb{Z}\star \mathbb{Z}^2) = \frac{1}{\sqrt{5}-2} \cong 4.23\dots$$

\subsection{Special cube complexes}

Associated to a given right-angled Artin group $\A(\Gamma)$ is an infinite and locally finite cube complex called the \emph{Salvetti complex}, constructed as follows: associated to each vertex of $\A(\Gamma)$ is a simple closed loop of length 1. If two vertices form an edge in $\A(\Gamma)$ then attach to the two associated loops a square torus generated by the two loops intersecting at right angle. More generally, given a complete subgraph on $k$ vertices, consider a unit $k$-torus generated by $k$ loops intersecting at right angles. We then take the universal cover of this tori-complex, which is called  
the Salvetti complex associated to $\A(\Gamma)$ and denoted as $X_{\mathbb{A}(\Gamma)}$. Notice that the 0 and 1-skeleta of $X_{\mathbb{A}(\Gamma)}$ are isomorphic, respectively, to 
the 0 and 1-skeleton of $Cay(\A(\Gamma))$, the Cayley graph of $\A(\Gamma)$ with this specific presentation.

Given a right-angled Artin group $\A(\Gamma)$, its Salvetti complex $X_{\mathbb{A}(\Gamma)}$ is a \emph{special cube complex} \cite{Haglund}, which we discuss now.  A \emph{cube complex} is a polyhedral complex in which the cells
are Euclidean cubes of side length one. The attaching maps are isometries identifying the
faces of a given cube with cubes of lower dimension and the intersection of two cubes is a
common face of each \cite{Sag95}. 

Cubes of dimension 0, 1 and 2 are also referred to as vertices, edges and squares. 
A cube complex is finite dimensional if there is an upper bound on the dimension of its cubes.
In this paper, we consider the distance between two vertices to be the graph distance on the 1-skeleton of a cube complex. 

A special cube complex  \cite{Haglund} is a particularly well-behaved class of cube complexes (for details see \cite{Haglund}). A \emph{CAT(0) cube complex} is a cube complex in which the link of each vertex is a flag simplicial complex.
 A (geometric) \emph{hyperplane} $h$ is a partition of the vertex set of  $X_{\mathbb{A}(\Gamma)}$ into two sets such that each set of the partition lie in a path-connected subspace of $X_{\mathbb{A}(\Gamma)}$ which we shall refer to as a \emph{half-space}, and denote as $\{ h^{+}, h^{-} \}$. Two hyperplanes provide four possible half-space intersections; the hyperplanes
\emph{intersect} if and only if each of these four half-space intersections is non-empty. If two hyperplanes do not intersect, we say they are \emph{nested}.

Two vertices in a half-space are connected by an edge-path that does not \emph{cross} $h$ whereas a geodesic edge-path
connecting a vertex in one half-space to a vertex in the other half-space crosses $h$ exactly once. In the latter case
we say that $h$ \emph{separates} the two vertices. Similar we say a geodesic (in the edge metric) between two vertices \emph{cross} a hyperplane $h$ if there exists two consecutive vertices on the geodesic such that one belongs to $h^{+}$ and the other belongs to $h^{-}$.  



\subsection{Geometric excursion in right-angled Artin groups} \label{S:geom}

Given $X_{\A(\Gamma)}$ with the edge metric, let $\calH_{A}$ be the set of hyperplanes dual to $a$, the generator of $A$. Let $x_{0} \in X_{\A(\Gamma)}$ be the base point given by the identity element. Given an element $g \in \AG$, consider a geodesic $\gamma$ in $X_{\A(\Gamma)}$ connecting $x_{0}$ and $g x_{0}$. By the theory of CAT(0) cube complexes, $\gamma$ crosses each hyperplane at most once. Given any two hyperplanes $h_{1}$ and $h_{2}$, their intersection is either empty or not. If their intersection is empty,  we say $h_{1}$ and $h_{2}$ \emph{nest}. Let $\{h_{i}^{+}, h_{i}^{-} \}$ denote the two half-spaces of each hyperplane $h_{i}$. Let $h_{1}$ and $h_{2}$ be two nested hyperplanes, and let us assume, without loss of generality, that $h_1^+ \supset h_{2}^{+}$. If there exists another hyperplane $h_{3}$ such that 
$$ h_1^+  \supset h_{3}^{\epsilon} \supset h_{2}^{+},$$
for some $\epsilon \in \{ \pm \}$,  
then we say that $h_{1}$ and $h_{2}$ are not \emph{consecutive hyperplanes}. Otherwise, if there does not exist such a hyperplane $h_{3}$,  then we say that $h_{1}$ and $h_{2}$ are \emph{consecutive hyperplanes}. A \emph{nested sequence} of hyperplanes is a sequence of hyperplanes 
\[
h_{1}, h_{2},...,h_{k}
\]
for which there exists an assignment of pluses and minuses so that $h_{i}^{+} \supset h_{i+1}^{+}$ for all $i = 1, 2, ..., k$. A \emph{maximal nested sequence} is a nested sequence of hyperplanes that is not a proper subsequence of another nested sequence.

It can be shown that each hyperplane in $X_{\A(\Gamma)}$ is dual to a set of edges associated to some vertex in $\Gamma$. That is to say, we can label each hyperplane with the corresponding vertex in $\Gamma$. This association is not bijective, as in general there are many hyperplanes associated with the same vertex in $\Gamma$. 

Given a finite geodesic segment $\gamma$ and a vertex group $A = \langle a \rangle$, we denote as $H_{\ga}|_{A}$ the set of hyperplanes labelled by $a$ that $\gamma$ crosses. It follows from the definition that $H_{\ga}|_{A}$ is a nested sequence of $a$-hyperplanes. We say that another nested sequence of hyperplanes is \emph{consistent} with $H_{\ga}|_{A}$ if it contains $H_{\ga}|_{A}$ as a subsequence.

We now give a definition of excursion of an element in a vertex subgroup in terms of hyperplanes.

\begin{definition}
Given an element $g \in \AG$ and a vertex subgroup $A = \langle a \rangle$, we define the \emph{geometric excursion} of $g$ in $A$ to be the maximal 
cardinality of a set of consecutive $a$-hyperplanes which separate $x_0$ and $g x_0$. 
 \end{definition}

 
%

\section{Growth series of graph products of groups} \label{S:growth-series}

\subsection{Amalgamated products and admissible embeddings}
A key tool in this paper is the growth series of an \emph{amalgamated product} of groups. Let $f_{i} \from K \to G_{i}$ for $i = 1, 2$ be group homomorphisms. The amalgamated product $G_{1}\star_{K}G_{2}$ is defined as follows: let $N$
be the normal subgroup of $G_{1} \star G_{2}$ generated by elements of the form $f_{1}(h)f_{2}(h)^{-1}$ for $h \in K$; then
\[G_{1} \star_{K} G_{2} := (G_{1} \star G_{2})/N. \]
Note that the free product $G_{1} \star G_{2}$ can be expressed as the special case of the amalgamated product where $K$ is trivial. 
As a result, we can view right-angled Artin groups as amalgamated products where one amalgam is a vertex group. 
Let $\fG \Gamma$ be a graph product with defining graph $\Gamma$ (in particular, it may be a right-angled Artin group). 
Then by definition for any vertex $v$ of $\Gamma$ we have the decomposition 
\begin{equation} \label{E:decomp}
\fG \Gamma = (\fG A \times  \fG B) \star_{\fG B} \fG Z
\end{equation}
where $A$ is a vertex $v$, $B$ is the subgraph of $\Gamma$ whose vertex set is the set of vertices adjacent to $v$, and $Z$ the subgraph of $\Gamma$ with vertex set 
the set of all vertices different from $v$.

 
In order to compute the growth series for an amalgamated product of groups, we need an additional assumption of \emph{admissibility}, which we introduce now, 
following \cite{Alo}. 
Recall $(G, S)$ is a group $G$ with symmetric generating set $S$. A \emph{map of pairs} 
\[ \alpha \from (K, S_K) \to (G, S)\] is a homomorphism $\alpha \from K \to G$ such that $\alpha(S_{K}) \subseteq S$.
The map $\alpha$ is an 
\emph{inclusion} if $\alpha \from K \to G$ is a monomorphism. Given $\alpha$, the 
corresponding length functions $\ell_{K}$ and $\ell_{G}$ satisfies 
\[
\ell_{G} (\alpha (k)) \leq \ell_{K} (k)
\]
 for all $k \in K$. 
\begin{definition}\label{admissible}
An inclusion $\alpha: K \to G$ of groups is \emph{admissible} if there exists a set $U$ of right coset representatives of $G$ modulo $K$ such that 
$$\ell_G(ku) = \ell_K(k) + \ell_G(u)$$
for any $k \in K, u \in U$. 
The set $U$ will be called a \emph{complement} of $K$ in $G$. In this paper we will assume that the identity element $1$ belongs to $U$. This implies that $K$ isometrically embeds into $G$; that is, for all $k \in K$ 
 one has the equality
$$\ell_{G} (\alpha(k)) = \ell_{K} (k).$$ 
 \end{definition}
 
As a consequence, every element $x \in G$ can be written uniquely in the form \[x = ku\] where $k \in K$ and $u \in U$. That is to say, as a set, $G$ is in bijection with a direct product. Then by Equation (1), the growth series behaves well, that is we have
$$G(t) = K(t) U(t).$$
Now, consider the inclusions $G_{i}  \to G_{1} \star_K G_{2}$ in the amalgamated product, for $i = 1, 2$. It is shown in 
\cite{Alo} that the admissibility of $K$ in $G_{i}$ implies the admissibility of $G_{i}  \to G_{1} \star_K G_{2}$.

\begin{lemma}[\cite{Alo}, Lemma 3]\label{Alonso}
If the inclusions $K \to G_{1}$ and $K \to G_{2}$ are admissible, then so are the inclusions $G_{1}  \to G_{1} \star_K G_{2}$ and $G_{2}  \to G_{1} \star_K G_{2}$. 
\end{lemma}

Using these techniques, we derive a normal form for any $x \in G$ as follows. 

Let $(K, S_K), (G_1, S_1), (G_2, S_2)$ be groups with their generating sets, and suppose that the maps $K \to G_1$ and $K \to G_2$ are admissible. Let $U_{1}, U_{2}$ be complements of $K$ in $G_{1}, G_{2}$ as in Definition~\ref{admissible}. Consider the amalgamated product
$$(G,S) = (G_{1}, S_{1}) \star_{(K, S_{K})} (G_{2}, S_{2}).$$
The following is a rephrasing of (\cite{Alo}, Lemma 3 and Proposition 4):

\begin{proposition}\label{normalform}
Every element $g \in G$ can be represented in a unique way as 
\begin{equation} \label{E:normal}
g = k u_{i_{1}} u_{i_{2}} \dots u_{i_{n}} 
\end{equation}
with $k \in K$, $u_{i_{j}} \in U_{i_{j}} \setminus \{1\}$, $i_{j}
 \in \{1, 2 \},\text{ and } i_{1} \neq i_{2} \neq ...\neq i_{n}$.
Moreover, its word length satisfies
\begin{equation} \label{E:length}
 \ell_{G}(g) = \ell_{K} (k) + \sum_{j =1}^{n} \ell_{G_{i_{j}}} (u_{i_{j}}).
\end{equation}
\end{proposition}

The proposition allows us to obtain the growth series of $G = G_1 \star_K G_2$, under the admissibility hypothesis. 
In fact, as a consequence of the proposition, we have the bijection 
$$G \cong K \times (U_1 \star U_2)$$
hence, since the above decomposition behaves nicely with respect to length, by equation \eqref{E:direct-prod} and equation \eqref{E:free-prod}, 
\begin{equation}\label{generalformforamalgam}
\frac{1}{G(t)} = \frac{1}{K(t)} \left( \frac{1}{U_1(t)} + \frac{1}{U_2(t)} - 1 \right) = \frac{1}{G_1(t)} + \frac{1}{G_2(t)} - \frac{1}{K(t)}
\end{equation}
as claimed in [\cite{Alo}, Theorem 2].

\medskip
Let us now apply this to graph products. Let $\fG \Gamma$ be a graph product with defining graph $\Gamma$. By (\cite{Chi}, Lemma 1), the inclusion of $\fG \Gamma'$ into $\fG \Gamma$ where $\Gamma'$ is an induced subgraph of $\Gamma$ is always admissible. Let us fix a vertex $v$ of $\Gamma$, and let $A, B, Z$ as in \eqref{E:decomp}. 
Applying the results and Proposition \ref{normalform} to the decomposition $(\fG A \times \fG B) \star_{\fG B} \fG Z$,
we conclude that all inclusions 
$$\xymatrix{ & \fG  A \times \fG B \ar[dr] & \\
\fG B \ar[ur] \ar[dr] & & (\fG A \times \fG B) \star_{\fG B} \fG Z \\
&\fG Z \ar[ur] & \\ }$$ are admissible and one can take $\fG A$ as a complement of $\fG B$ in $\fG A \times \fG B$. Let us now choose a complement of $\fG B$ in $\fG Z$
and denote it as $T$. Then, the normal form for any $g \in \A(\Gamma)$ is 
\begin{equation} \label{E:norm-A}
g_{\rm norm} := b a_{1} t_{1} a_{2} t_{2}\dots a_k t_k 
\end{equation}
where $b \in \fG B$, $a_{i} \in \fG A$ and $t_{i} \in T$, with $a_i \neq 1$ for $i > 1$ and $t_i \neq 1$ for $i < k$.  

\begin{remark}\label{inbetween}
By construction, in $g_{\rm norm}$ every subword $t_{i}$ contains a generator of $\A(\Gamma)$ that does not commute with $a \in \fG A $.
\end{remark}

In the case of a graph product $\fG \Gamma$ with defining graph $\Gamma$ and vertex groups $\{ G_v \}_{v \in V}$, we get by induction the formula (\cite{Chi}, Proposition 1)
\begin{equation} \label{E:gen-fun-noexc}
\frac{1}{\fG \Gamma(t)} = \sum_{\Delta} \prod_{v \in \Delta} \left(\frac{1}{G_v(t)}-1 \right) 
\end{equation}
where the sum is over all complete subgraphs $\Delta$ of $\Gamma$.

\medskip
\textbf{Example.} Let us consider the right-angled Artin group $\fG \Gamma = \mathbb{Z} \star \mathbb{Z}^2$. Then its defining graph $\Gamma$ has $3$ vertices and $1$ edge, and all its vertex groups isomorphic to $\mathbb{Z}$. Hence 
$$\frac{1}{\fG \Gamma(t)} = 1 + 3 \left( \frac{1}{\mathbb{Z}(t)} - 1\right) + \left( \frac{1}{\mathbb{Z}(t)} - 1\right)^2 = \frac{1 -4t - t^2}{(1+t)^2}$$
as we previously obtained. In general, note that as a corollary of formula \eqref{E:gen-fun-noexc} if $G = \mathbb{A}(\Gamma)$ is a RAAG, then 
$$\frac{1}{G(t)} = p\left( \frac{-2t}{1+t} \right)$$
where $p$ is the clique polynomial of $\Gamma$ (\cite{Chi}).

\subsection{Excursion in graph products}

Let $\fG$ be a graph product, and let $v \in \Gamma$ be a vertex. Recall the decomposition 
$$\fG \Gamma = (\fG A \times \fG B) \star_{\fG B} \fG Z$$
where $A$ is the vertex $v$, $B$ is the subgraph consisting of all vertices which are adjacent to $A$, and $Z$ the subgraph of all vertices different from $A$.

Let $g \in \fG \Gamma$, and recall the normal form from eq. \eqref{E:norm-A}. We define the excursion of $g$ in the vertex group of $v$ as 
the maximal length of the words $a_i \in \fG A$ in the normal form of $g$: 
$$\mathcal{E}_v(g) := \max_{1 \leq i \leq k} \Vert a_i \Vert.$$

We now derive the growth series for the subset of a graph product where the excursion in a given vertex group $G_v$ is 
bounded above by some constant $E$.  

\begin{proposition} \label{P:formula}
Let $\fG \Gamma$ be a graph product, let $v$ be a vertex, and let $\calE_v(g)$ denote the excursion of $g$ in the vertex group of $v$. 
Define for any positive integer $E \in \N$
$$\fG \Gamma_{E} := \{ g \in \fG \Gamma \ : \ \calE_v(g) \leq E \}.$$
Then its growth series satisfies
\begin{equation} \label{E:gen-fun}
\frac{1}{\fG \Gamma_E(t)} = \sum_{\Delta'} \prod_{w \in \Delta '} \left(  \frac{1}{G_w(t)} - 1 \right) \left(  \frac{1}{\fG A_E(t)} - 1 \right)  + \sum_{\Delta''} \prod_{w \in \Delta''} \left(  \frac{1}{G_w(t)} - 1 \right)
\end{equation}
where $G_w$ is the vertex group of $w$, $\Delta'$ runs through all complete subgraphs of $B$, and $\Delta''$ runs through all complete subgraphs of $Z$. 
\end{proposition}

\textbf{Example.} Let us consider the group $\fG \Gamma= \mathbb{Z}\star \mathbb{Z}^2$, where $\Gamma$ has three vertices and one edge;  
let $v$ be the isolated vertex in $\Gamma$, and let $\fG A = \langle v \rangle \cong \mathbb{Z}$ the cyclic group generated by $v$. 
Then the graph $B$ is empty, $Z$ is the complete graph on two vertices, and all vertex groups are $G_w \cong \mathbb{Z}$. Then the set of elements in $\fG \Gamma$ which have excursion $\leq E$ in the vertex group $\fG A$ is given, according to the formula \eqref{E:gen-fun} above, by 
$$\frac{1}{(\mathbb{Z}\star \mathbb{Z}^2)_E (t)} = 1 \cdot \left( \frac{1}{\mathbb{Z}_E(t)} - 1\right) + 1 + 2 \left( \frac{1}{\mathbb{Z}(t)} - 1\right) + \left( \frac{1}{\mathbb{Z}(t)} - 1\right)^2.$$
Now, it is easy to check that 
$$\mathbb{Z}_E(t) = 1 + 2 \sum_{k = 1}^E t^k = \frac{1+t - 2t^{E+1}}{1-t} \qquad\textup{and } \mathbb{Z}(t) = \frac{1+t}{1-t}$$
hence, after some simplification, 
$$(\mathbb{Z}\star \mathbb{Z}^2)_E (t) = \frac{(1+t)^2 (1+t - 2t^{E+1})}{(1-4t-t^2)(1+t) + 8 t^{E+2}}.$$
Note in particular that as $E \to \infty$ we obtain the growth series of $\mathbb{Z} \star \mathbb{Z}^2$ as in equation \eqref{E:Z2}.


\begin{proof}
Let $A, B, Z$ as above. Since $\fG \Gamma = G_1 \star_{\fG B} G_2$ where $G_1 = \fG A \times \fG B$ and $G_2 = \fG Z$, then 
we can take $A = U_1$ as complement of $\fG B$ in $G_1$, and let $U_2$ be a complement of $\fG B$ in $\fG Z$. 
Consider an element $x \in  \fG \Gamma$: this has normal form 
$$x = b u_{i_1} \dots u_{i_k}$$
where $b \in \fG B$, $u_{i_j} \in U_{i_j} \setminus \{1 \}$, and $i_j \neq i_{j+1}$. 
Then its image in $\fG A \star \fG Z$ is
$$p(x) = u_{i_1} \dots u_{i_k}$$
as $p(b) = 1$ by definition. 
Hence one has the bijection (note that this is not a group homomorphism!)
$$\fG \Gamma_E \to \fG B \times (\fG A_E \star U_2).$$
Hence, if we fix an upper bound $E$ on the excursion, 
the generating function for the set of elements of $G\Gamma$ with excursion $\leq E$ satisfies
$$\frac{1}{\fG \Gamma_E(t)} = \frac{1}{\fG B(t)} \left( \frac{1}{\fG A_E(t)} + \frac{1}{U_2(t)} - 1\right) = \frac{1}{\fG B(t)} \left( \frac{1}{\fG A_E(t)} - 1 \right) + \frac{1}{\fG Z(t)} $$
and by using \cite{Chi} and the fact that $\fG Z(t) = \fG B(t) U_2(t)$ we have 
$$ = \sum_{\Delta'} \prod_{w \in \Delta '} \left(  \frac{1}{G_w(t)} - 1 \right) \left(  \frac{1}{\fG A_E(t)} - 1 \right)  + \sum_{\Delta''} \left(  \frac{1}{G_w(t)} - 1 \right)
$$
as desired. 
\end{proof}

\subsection{Equivalence of definition of excursions in RAAGs}

We now focus on the special case of right-angled Artin groups, and show that the three definition of excursion given coincide. 

\begin{definition}
The \emph{algebraic excursion} of an element of $g \in \A(\Gamma)$ in a vertex group $\fG A = \langle a \rangle$ is the maximal number of consecutive $a$'s in any 
geodesic word representing $g \in \A(\Gamma)$. 
\end{definition}
It is clear from the definition that this is a special case of the definition of the excursion given in the introduction. We now prove the following equivalence.

\begin{proposition}\label{equal}
Let $\A(\Gamma)$ be a right-angled Artin group with a vertex group $\fG A = \langle a \rangle$. 
Given $g \in \A(\Gamma)$, let $g_{\rm norm}$ denote the normal form of $g$ as in Definition \ref{normalform}. The following are equal:
\begin{enumerate}
\item the maximal number of consecutive $a$'s in $g_{\rm norm}$;
\item the algebraic excursion of $g$ in $\fG A$;
\item the geometric excursion of $g$ in $\fG A$.
\end{enumerate}
\end{proposition}
\begin{proof}
$(1) \leq (2)$: (2) is maximized over all reduced words representing $g$, while $(1)$ is counting the number of consecutive $a$'s in $g_{\rm norm}$, thus $(1) \leq (2)$. 

$(2) \leq (3)$: Let $w$ be a geodesic word which represents $g$ and realizes the algebraic excursion. Then any sequence of consecutive $a$'s in the word $w$
gives rise to a sequence of consecutive dual hyperplanes which separate $x_0$ and $g x_0$, which proves $(2) \leq (3)$. 

$(3) \leq (1)$: Given a sequence $H$ of nested, consecutive $a$-hyperplanes that separate $x_0$ and $g x_0$ and realizes the maximum in (3), consider the
geodesic path $\gamma$ between $x_{0}$ and $g x_0$ that represents the normal form $g_{\rm norm}$. Let us consider the sequence $H_{\ga}|_{A}$ 
of $a$-hyperplanes which are dual to $\gamma$. 
Since each hyperplane separates the space, and half-spaces are convex, it follows that $H_{\ga}|_{A}$ contains $H$. \\
\noindent \textbf{Claim}: Any two consecutive $a$-hyperplanes $h_{1}$ and $h_{2}$ in $H_{\ga}|_{A}$ belong to the same subword $a_i$ in $g_{\rm norm}$ (same notation as in eq. \eqref{E:norm-A}).
\subsubsection*{Proof of Claim:} otherwise by Remark~\ref{inbetween}, there exists, in between the two letters corresponding to $h_1$ and $h_2$, a letter that is a generator in $Z \setminus E$ and does not commute with $a$. 
This implies that there exists a hyperplane $h$ that is nested and in between $h_{1}$ and $h_{2}$, which contradicts the choice that $h_{1}$ and $h_{2}$ are consecutive. 
It follows from the claim that $(3) \leq (1)$.
\end{proof}

Thus, we define the excursion of an element $g$ in a subgroup generated by a vertex $v$, denoted as $\calE_{v}(g)$, in any of the three equivalent ways given by 
Proposition~\ref{equal}

\section{Proof of the main theorem} \label{S:proof}

Let us now turn to the proof of Theorem \ref{T:main}. First of all, we need a few results on families of generating functions which depend on a parameter. 

\subsection{Growth of sequences and generating functions}
\begin{proposition} \label{P:gen-funE}
For each $E \in \mathbb{N} \cup \{ \infty \}$, let 
$$f_E(t) := \frac{P_E(t)}{Q_E(t)}$$  
be a sequence of rational functions. Let
$$f_E(t) = \sum_{n = 0}^\infty a_{n, E} \ t^n$$
be the Taylor expansion of $f_E(t)$ at the origin, and assume that $a_{n, E} \geq 0$ for all $E, n$. 
Moreover, assume that: 
\begin{enumerate}
\item
$$P_E(t) \to P_\infty(t) \qquad \textup{ and } \qquad Q_E(t) \to Q_\infty(t)$$ 
uniformly on compact subsets of $\{t\in \mathbb{C} \ : \  |t| < 1 \}$. 
\item
The limit function $Q_\infty(t)$ has only one root of smallest modulus at $t = r < 1$ and this root is simple. 
\item 
The numerator $P_\infty(t)$ satisfies $P_\infty(r) \neq 0$.  
\end{enumerate}
Then: 
\begin{enumerate}
\item if we denote as $\lambda_E^{-1}$ the smallest root of $Q_E(t)$, then 
$$\lambda_E \to \lambda$$
\item
there exists $C > 0, E_0,$ and $n_0$ such that for any $E \geq E_0$ and any $n \geq n_0$ 
$$C^{-1} \leq \frac{a_{n, E}}{(\lambda_E)^n} \leq C.$$
\item
Moreover, there exists $c > 0$ such that 
$$\lim_{\min\{n, E \}\to \infty} \frac{a_{n, E}}{(\lambda_E)^n} = \lim_{n \to \infty} \frac{a_n}{\lambda^n} = c.$$
\end{enumerate}

\end{proposition}


To prove the Proposition, we will use the following basic Lemma which relates the growth of a sequence to the smallest pole of 
its generating function.

\begin{lemma} \label{L:gt}
Let 
$$f(t) = \sum_{n = 0}^\infty a_n \ t^n$$ 
be a generating function with $a_n \geq 0$, and suppose that one can write 
$$f(t) = \frac{g(t)}{1 - \lambda t}$$
where $g$ is holomorphic in a disk $\{ t \in \mathbb{C} \ : \ |t| \leq R\}$, and let $R'$ be a constant which satisfies $\lambda^{-1} < R' < R < 1$. 
Then for any $n$ one has the inequality
$$\left| \frac{a_n}{\lambda^n} - g(\lambda^{-1}) \right| \leq \sup_{|t| = R} |g(t)| \cdot \frac{(R'/R)^{n+1}}{1 - R'/R}.$$ 
\end{lemma}

\begin{proof}
Let us denote as $g(t) := \sum_{n = 0}^\infty g_{n} t^n$ the Taylor series of $g(t)$. 
Then 
$$a_{n} = \sum_{k = 0}^n g_{k} \lambda^{n-k}$$
hence 
$$\frac{a_{n}}{\lambda^n} = \sum_{k = 0}^n g_{k} \lambda^{-k} = g(\lambda^{-1}) - \sum_{k = n+1}^\infty g_k \lambda^{-k}.$$
Let us recall that Cauchy's integral formula yields
$$|g_k| = \left| \frac{1}{2 \pi i} \int_{|t| = R} \frac{g(t)}{t^{k+1}} \ dt \right| \leq  
\frac{1}{2 \pi} \int_{|t| = R} \frac{ |g(t)| }{R^{k+1}} \ |dt| \leq \frac{ \sup_{|t| = R} |g(t)| }{R^{k}}$$
hence 
$$\left| \frac{a_{n}}{\lambda^n} - g(\lambda^{-1}) \right| \leq   \sup_{|t| = R} |g(t)| \cdot \sum_{k = n+1}^\infty (R \lambda)^{-k} \leq 
\sup_{|t| = R} |g(t)| \cdot \frac{(R'/R)^{n+1}}{1 - R'/R}$$
as claimed.
\end{proof}

\begin{proof}[Proof of Proposition \ref{P:gen-funE}] 

Let us choose $R$ such that $r < R < 1$ and that $Q_\infty(t)$ does not have any zero of modulus less than $R$ other than $r$. Then for $E$ sufficiently large, in the disk $\{t \in \mathbb{C} \ : \ |t| < R\}$ there is exactly one root of $Q_E(t)$, counted with multiplicity. Let us call such a root $r_E$. By taking smaller and smaller disks around $r$, we find out that $r_E \to r$ as $E \to \infty$. Let us also pick $R' < R$ such that $|r_E| < R'$ for all $E$. 

Let us define, for $E \in \mathbb{N} \cup \{ \infty \}$, $g_E(t) := f_E(t)(1 - \lambda_E t)$. Since $Q_\infty$ does not have roots on $|t| = R$, then $g_E(t) \to g_\infty(t)$ uniformly on $|t| = R$, 
hence for $E$ sufficiently large $\sup_{|t| = R} |g_E(t)| \leq \sup_{|t| = R} |g_\infty(t)| + 1$. 

Moreover, let us write 
$$Q_E(t) = \widetilde{Q}_E(t) (1 - \lambda_E t)$$
hence by taking its derivative 
$$Q'_E(t) = \widetilde{Q}'_E(t) (1 - \lambda_E t) +\widetilde{Q}_E(t) (- \lambda_E)$$
thus 
$$\widetilde{Q}_E( \lambda_E^{-1})  = - \frac{Q_E'(\lambda_E^{-1})}{\lambda_E}.$$
Hence, since $Q_\infty'(\lambda^{-1}) \neq 0$ and $P_\infty(\lambda^{-1}) \neq 0$, we have 
$$g_E(\lambda_E^{-1}) = \frac{P_E(\lambda_E^{-1})}{\widetilde{Q}_E(\lambda_E^{-1})} \to g_\infty(\lambda^{-1}) \neq 0.$$
Now, from Lemma \ref{L:gt} there exists a constant $s < 1$ such that for any $E$ and any $n$
$$\left| \frac{a_{n, E}}{(\lambda_E)^n} - g_E(\lambda_E^{-1}) \right| \leq s^n$$
and
$$\left| \frac{a_n}{\lambda^n} - g(\lambda^{-1}) \right| \leq s^n$$
hence 
$$\lim_{\min\{n, E\} \to \infty} \left|\frac{a_{n, E}}{(\lambda_E)^n} - g(\lambda^{-1}) \right| = 0$$
which proves the claim. 
\end{proof}



\subsection{Proof of Theorem \ref{T:main} for vertex groups}

Now we are ready to state and prove our main theorem. Recall that $P_n$ is the uniform probability on the set of group elements of length $n$. 

\begin{theorem} \label{T:main-exc}
Let $G = \A(\Gamma)$ be a non-elementary, irreducible right-angled Artin group, let $\lambda > 1$ 
 be its growth rate, and let $v$ be a vertex of $\Gamma$. Then for any $\epsilon > 0$ the probability of a given excursion in the vertex group of $v$ satisfies 
$$P_n\left( g \in G \ : \ \frac{1}{\log \lambda} - \epsilon \leq \frac{\mathcal{E}_v(g)}{\log n} \leq \frac{1}{\log \lambda} \right) \to 1$$
as $n \to \infty$.
\end{theorem}

The proof is based on the study of the followed modified generating function. Given $G = \A(\Gamma)$ and $E \geq 0$, we define the generating function 
$$G_E(t) = \sum_{n = 0}^\infty a_{n , E} t^n$$
where $a_{n, E} = \# \{ g \in G \ : \  \Vert g \Vert = g, \mathcal{E}_v(g) \leq E \}$ is the number of elements of length $n$ in $G$ which have excursion in $G_v$ at most $E$, and we denote the usual growth series of $G$ as 
$$G(t) = \sum_{n = 0}^\infty a_{n} t^n$$
where $a_n = \#\{ g \in G \ : \ \Vert g \Vert = n \}$. Then, the Patterson-Sullivan measure of the set of elements with excursion $\leq E$ equals 
$$\lim_{n \to \infty} \frac{a_{n , E}}{a_n}$$
hence our goal is to study this ratio as $n$ and $E$ vary, and prove that there exists a limit if $E \approx \log n$. 

\medskip

By formula \eqref{E:gen-fun-noexc}, the growth series $G(t)$ is a rational function $G(t) = \frac{P(t)}{Q(t)}$. Since $G$ is irreducible, by Theorem \ref{T:exact} 
it has exact exponential growth $\lambda > 1$; thus, the denominator $Q(t)$ has a root at $r = \frac{1}{\lambda}$, such a root is simple, and there are no other roots of modulus $\leq r$.  

Now, in the case of right-angled Artin groups, formula \eqref{E:gen-fun} becomes 
\begin{equation} \label{E:RAAG}
\frac{1}{\A(\Gamma)_E(t)} = \sum_{\Delta'} \left( \frac{-2t}{1+t} \right)^{|\Delta'|} \left(  \frac{1}{\mathbb{Z}_E(t)}-1 \right) + \sum_{\Delta''}  \left( \frac{-2t}{1+t} \right)^{|\Delta''|} = \frac{Q_E(t)}{P_E(t)}
\end{equation}
and we note that 
\begin{equation} \label{E:bi}
\frac{1}{\mathbb{Z}_E(t)}-1 = \frac{-2t + 2t^{E+1}}{1 + t - 2 t^{E+1}} = \frac{b^{(1)}_E(t)}{b^{(2)}_E(t)}
\end{equation}
hence $P_E(t)$ can be taken of the form $P_E(t) = (1+t)^N (1 + t - 2 t^{E+1})$ for some $N$, which implies that $P(t) = \lim_{E \to \infty} P_E(t)$ 
does not vanish in the unit disk. 


\begin{lemma} \label{L:lambda}
There exist $C_1, C_2 > 0$ such that for all $E$
$$C_2 \lambda^{-E} \leq |\lambda - \lambda_E| \leq C_1 \lambda^{-E}.$$
\end{lemma}

\begin{proof}
By elementary calculus, 
$$Q_E(\lambda^{-1}) - Q_E(\lambda_E^{-1}) = Q'_E(\xi) (\lambda^{-1} - \lambda_E^{-1}) = \frac{Q'_E(\xi) }{\lambda \lambda_E}(\lambda_E - \lambda)$$
with $\xi \in (\lambda^{-1}, \lambda_E^{-1})$. On the other hand, since $Q_E(\lambda_E^{-1}) = 0 = Q_\infty(\lambda^{-1})$, then 
$$Q_E(\lambda^{-1}) - Q_E(\lambda_E^{-1}) = Q_E(\lambda^{-1}) - Q_\infty(\lambda^{-1})$$
hence 
$$| \lambda - \lambda_E | = \frac{\lambda \lambda_E}{Q'_E(\xi)} |Q_E(\lambda^{-1}) - Q_E(\lambda_E^{-1})|$$
and since $r = \frac{1}{\lambda}$ is a simple root of $Q_\infty(t)$, then there exist a constant $c > 0$ such that 
$$\frac{1}{c} \leq \frac{|Q_E'(\xi)|}{\lambda \lambda_E} \leq c.$$
Now, by expanding the terms in formula \eqref{E:RAAG}, if we denote $\fG B(t) = \frac{r_1(t)}{r_2(t)}$ and $\fG Z(t) = \frac{s_1(t)}{s_2(t)}$ with $r_i(t)$ and $s_i(t)$ 
polynomials, we get 
$$Q_E(t) = r_2(t) s_1(t) b^{(1)}_E(t) + r_1(t) s_2(t) b^{(2)}_E(t).$$
Now, by \eqref{E:bi}, for $i = 1,2$ we have 
$$b^{(i)}_E(t) - b^{(i)}_\infty(t) = \pm 2 t^{E+1}$$
then also 
$$|Q_E(t) - Q_\infty(t)| = 2(r_2(t) s_1(t) - s_2(t) r_1(t))t^{E+1}$$
hence 
$$| \lambda - \lambda_E | \leq c \lambda^{-E}.$$
Now, since $B \subseteq Z$, we have $\fG B(t) \leq \fG Z(t)$ for any $t \geq 0$. Moreover, since $\A(\Gamma)$ is irreducible, then 
$B \subsetneq Z$, hence $\fG B(t) < \fG Z(t)$ for any $t > 0$. This implies that $r_2(r) s_1(r) \neq r_1(r) s_2(r)$, hence we also have 
$$| \lambda - \lambda_E | \geq c \lambda^{-E}$$
as claimed.
\end{proof}

\begin{lemma}  \label{L:liminf}
If $c  = \frac{1}{\log \lambda}$ then 
$$0 < \liminf_{n \to \infty} \left( \frac{\lambda_{c \log n}}{\lambda} \right)^n \leq \limsup_{n \to \infty} \left( \frac{\lambda_{c \log n}}{\lambda} \right)^n < 1.$$
\end{lemma}

\begin{proof}
From Lemma \ref{L:lambda}, for any $c > 0$ and any $n$ sufficiently large, 
$$\lambda - \lambda_{c \log n} \approx \lambda^{-c \log n} = n^{-c \log \lambda}$$
hence since $c = \frac{1}{\log \lambda}$
$$ \left( \frac{\lambda_{c \log n}}{\lambda} \right)^n \geq \left( \frac{\lambda - C n^{- c \log \lambda} }{\lambda} \right)^n  = 
\left(1 - \frac{C}{\lambda n} \right)^n \to e^{-C/\lambda} > 0$$
and similarly for the upper bound, yielding the claim. 
Moreover, the same computation yields for $c > \frac{1}{\log \lambda}$
$$\lim_{n \to \infty} \left( \frac{\lambda_{c \log n}}{\lambda} \right)^n = 1$$
and for $c < \frac{1}{\log \lambda}$
$$\lim_{n \to \infty} \left( \frac{\lambda_{c \log n}}{\lambda} \right)^n = 0.$$
\end{proof}

By Proposition \ref{P:gen-funE} applied to $E = c \log n$ we have 
$$P_n(g \ : \ \mathcal{E}_v(g) \leq c \log n) = \frac{a_{n, c \log n}}{a_n} \sim \left( \frac{\lambda_{c \log n}}{\lambda} \right)^n$$ 
where $\sim$ means that the terms are asymptotic (i.e. their ratio tends to $1$ as $n \to \infty$).
Hence, by Lemma \ref{L:liminf} we get that if $c < \frac{1}{\log \lambda}$ then 
$$\lim_{n \to \infty} P_n\left( \mathcal{E}_v(g) \leq c \log n \right) = 0$$
while if $c \geq \frac{1}{\log \lambda}$ then 
$$\lim_{n \to \infty} P_n\left( \mathcal{E}_v(g) \leq c \log n \right) = 1$$
which completes the proof of Theorem \ref{T:main-exc}. 

\subsection{Excursion in flats}

Let us now extend our results from the excursion in vertex groups to excursion in flats. 

Given a subgroup $H < G$, let us define $\mathcal{E}_H(g)$ as the maximal distance traveled in a coset of $H$ by a geodesic path from $1$ to $g$.
The following lemma relates excursion in a flat to excursion in its vertex subgroups. 

\begin{lemma} \label{L:higher}
Let $G = \mathbb{A}(\Gamma)$, and $H$ be a rank $d$ abelian subgroup of $G$, generated by vertices $v_1, \dots, v_d$. 
Then for any $g \in G$
$$\max_{1 \leq i \leq d} \mathcal{E}_{v_i}(g) \leq \mathcal{E}_H(g) \leq d \max_{1 \leq i \leq d} \mathcal{E}_{v_i}(g).$$
\end{lemma}

\begin{proof}
Suppose that the path $\gamma$ joining $1$ and $g$ has excursion $N$ in $v_i$. Then it has also excursion $\geq N$ in $H$, since $\langle v_i \rangle$ is a 
subgroup of $H$. This proves the left-hand side inequality. 

On the other hand, suppose that the path $\gamma$ has excursion $N$ in $H$. This means that there exists a flat and a path of length $N$ which stays inside the flat. 
Since the length of the path is the sum of the length of the $d$ projections of the path in directions $v_1, \dots, v_d$, then there exists an index $i$ 
such that $\mathcal{E}_{v_i}(\gamma) \geq \frac{N}{d}$. Hence $\max \mathcal{E}_{v_i}(\gamma) \geq \frac{1}{d} \mathcal{E}_H(\gamma)$, which proves the 
inequality on the right-hand side. 
\end{proof}

\begin{proof}[Proof of Theorem \ref{T:main} for higher dimensional flats]
By Theorem \ref{T:main-exc}, there exists $c = \frac{1}{\log \lambda} - \epsilon > 0$ such that 
$$P_n(\mathcal{E}_{v_1}(g) \leq  c \log n) \to 0.$$
Then 
$$P_n(\mathcal{E}_H(g) \leq c \log n ) \leq P_n(\mathcal{E}_{v_1}(g) \leq c \log n) \to 0.$$
On the other hand, there exists $c' = \frac{1}{\log \lambda} > 0$ such that $P_n( \mathcal{E}_{v_i}(g) \leq c' \log n ) \to 1$ for any $i = 1, \dots, d$. 
Then by Lemma \ref{L:higher}
$$P_n(\mathcal{E}_H(g) \leq c' d \log n) \geq P_n \left(\max_{1 \leq i \leq d} \mathcal{E}_{v_i}(g) \leq c' \log n \right) = P_n \left( \bigcap_{i = 1}^d \left\{ \mathcal{E}_{v_i}(g) \leq c' \log n \right\} \right) \to 1,$$
completing the proof of Theorem \ref{T:main} for higher dimensional flats.
\end{proof}

\subsection{The reducible case}

In the reducible case, various things can happen. If $\A(\Gamma)$ is reducible, then $\Gamma^{op}$ is disconnected. 
Let us fix a vertex $v$, and denote as $\Gamma_1^{op}$ the connected component 
of $v$ in $\Gamma^{op}$, and let $\Gamma_2^{op} = \Gamma^{op} \setminus \Gamma_1^{op}$ be the complementary graph. Then we have a canonical 
isomorphism 
$$\varphi: \A(\Gamma) \to \A(\Gamma_1) \times \A(\Gamma_2)$$
where $\A(\Gamma_1)$ is irreducible. Let $\lambda$ be the growth rate of $\A(\Gamma)$, and $\lambda_1$ the growth rate of $\A(\Gamma_1)$. 
Then by definition if $g \in \A(\Gamma)$ and $\varphi(g) = (g_1, g_2)$, the excursion satisfies $\mathcal{E}_v(g) = \mathcal{E}_v(g_1)$.
There are two cases.

\smallskip
\textbf{Case 1:} $\lambda_1 = \lambda > 1$. Then the generic excursion in $v$ is still logarithmic. 

For each $i = 1, 2$, denote as $a^{(i)}_n$ the cardinality of the set of elements of length $n$ in $\A(\Gamma_i)$, and as $a^{(i)}_{n, l}$ 
the cardinality of the set of elements in $\A(\Gamma_i)$ of length $n$ and with excursion $\leq l$ in the vertex group of $v$. 
Then 
$$a_{n, c \log n} = \sum_{k = 0}^{n} a^{(1)}_{k, c \log n} a^{(2)}_{n-k}.$$
First, we claim
$$\lim_{n \to \infty} \frac{ \sum_{k = 0}^{\sqrt{n}} a_k^{(1)} a_{n-k}^{(2)} }{a_n} = 0.$$
\begin{proof}[Proof of the Claim.]
If $\lambda_2 < \lambda_1 = \lambda$, then $a^{(2)}_n \lesssim (\lambda - \epsilon)^n$ for some $\epsilon > 0$, and $a_n \approx a_n^{(1)} \approx \lambda^n$. Then 
$$\frac{1}{a_n} \sum_{k = 0}^{\sqrt{n}-1} a_k^{(1)} a_{n-k}^{(2)} \lesssim \frac{1}{\lambda^n} \sum_{k = 0}^{\sqrt{n} -1} \lambda^k (\lambda - \epsilon)^{n-k} \lesssim
 \sum_{k = 0}^{\sqrt{n}-1} \left( 1 - \frac{\epsilon}{\lambda} \right)^{n-k} \lesssim \sqrt{n} \left(1 - \frac{\epsilon}{\lambda} \right)^{n - \sqrt{n}} \to 0.$$
If $\lambda_2 = \lambda_1 = \lambda$, then let $a_n^{(2)}$ is of the form $a_n^{(2)} \approx n^s \lambda^n$ for some $s \geq 0$. Then one has $a_n \approx n^{s+1} \lambda^n$, hence 
$$\frac{1}{a_n} \sum_{k = 0}^{\sqrt{n}-1} a_k^{(1)} a_{n-k}^{(2)} \lesssim \frac{1}{n^{s+1} \lambda^n} \sum_{k = 0}^{\sqrt{n}-1} \lambda^k \lambda^{n-k} (n-k)^s
\lesssim \frac{\sqrt{n} \cdot n^s}{n^{s+1}} = \frac{1}{\sqrt{n}} \to 0.$$
\end{proof}
Note that from the claim follows that 
\begin{equation} \label{E:full}
\lim_{n \to \infty} \frac{ \sum_{k = \sqrt{n}}^{n} a_k^{(1)} a_{n-k}^{(2)} }{a_n} = 1.
\end{equation}
Now, let us note that if $\sqrt{n} \leq k \leq n$, then $\log k \leq \log n \leq 2 \log k$ hence  
\begin{equation}\label{E:reducible1}
\sum_{k = \sqrt{n}}^{n} a^{(1)}_{k, c \log n} a^{(2)}_{n-k} \leq \sum_{k = \sqrt{n}}^{n} \frac{a^{(1)}_{k, 2 c \log k}}{a^{(1)}_{k}}   a^{(1)}_{k} a^{(2)}_{n-k} \leq \left( \sup_{k \geq \sqrt{n}}  \frac{a^{(1)}_{k, 2 c \log k}}{a^{(1)}_{k}}  \right)
 \sum_{k = \sqrt{n}}^{n}  a^{(1)}_{k} a^{(2)}_{n-k}. 
\end{equation}
On the other hand, 
\begin{equation} \label{E:reducible2}
 \sum_{k = \sqrt{n}}^{n} a^{(1)}_{k, c \log n} a^{(2)}_{n-k} \geq \sum_{k = \sqrt{n}}^{n} \frac{a^{(1)}_{k, c \log k}}{a^{(1)}_{k}}   a^{(1)}_{k} a^{(2)}_{n-k} \geq \left( \inf_{k \geq \sqrt{n}}  \frac{a^{(1)}_{k, c \log k}}{a^{(1)}_{k}}  \right) \sum_{k = \sqrt{n}}^{n}   a^{(1)}_{k} a^{(2)}_{n-k}. 
\end{equation}
Note now that for any $c > 0$ one can write
$$\frac{a_{n, c \log n}}{a_n} = \frac{1}{a_n} \sum_{k = 0}^{\sqrt{n}} a^{(1)}_{k, c \log n} a^{(2)}_{n-k} + \frac{1}{a_n} \sum_{k = \sqrt{n}}^{n} a^{(1)}_{k, c \log n} a^{(2)}_{n-k} = 0.$$
Now, by the claim the first term tends to $0$. Moreover, by the irreducible case we can find $c > 0$ such that $\lim \frac{a_{n, 2 c \log n}}{a_n} \to 0$, hence 
using \eqref{E:full} and \eqref{E:reducible1} we get 
$$\frac{a_{n, c \log n}}{a_n} \to 0.$$
Similarly, we can find $c' > 0$ such that $\lim \frac{a_{n, c' \log n}}{a_n} \to 1$, hence by \eqref{E:full} and \eqref{E:reducible2} we have 
$$\frac{a_{n, c' \log n}}{a_n} \to 1,$$
which completes the proof.

\smallskip
\textbf{Case 2:} $\lambda_1 < \lambda$. Then the generic excursion is no longer logarithmic. In fact, we have 
$$a_{n, 0} \geq a_{n}^{(2)}$$
and $\liminf_{n \to \infty} \frac{a_n^{(2)}}{a_n} > 0$, hence 
$$\liminf_{n \to \infty} P_n(\mathcal{E}_v(g) = 0) = \liminf_{n \to \infty}  \frac{a_{n, 0}}{a_n} \geq  \liminf_{n \to \infty}  \frac{a_{n}^{(2)}}{a_n} > 0.$$
Thus, there is a positive asymptotic probability that the excursion is zero, hence it is not generically logarithmic in $n$. To sum up, we have proved the following result.

\begin{theorem} \label{T:all-cases}
Let $G = \A(\Gamma)$ be a non-abelian right-angled Artin group, let $v$ be a vertex of $\Gamma$ and let $\mathbb{A}(\Gamma_v) \subseteq \A(\Gamma)$ be the 
irreducible right-angled Artin group associated to the component of $v$. 
Then: 
\begin{enumerate}
\item if the growth rate of $\mathbb{A}(\Gamma_v)$ equals the growth rate of $\A(\Gamma)$, then the excursion $\mathcal{E}_v(g)$ is asymptotically logarithmic; 
that is, there exist constants $0 < c_1 < c_2$ such that 
$$P_n(g \in G \ : \  c_1 \log n \leq \mathcal{E}_v(g) \leq c_2 \log n) \to 1$$
as $n \to \infty$;
\item if the growth rate of $\mathbb{A}(\Gamma_v)$ is smaller than the growth rate of $\A(\Gamma)$, then 
$$\liminf_{n \to \infty} P_n(g \in G \ : \ \mathcal{E}_v(g) = 0 ) = c > 0.$$
\end{enumerate}
\end{theorem}

As a corollary, this immediately implies Theorem \ref{T:irred} in the introduction: 

\begin{corollary}
A non-elementary right-angled Artin group $\A(\Gamma)$ is irreducible if and only if it has pure exponential growth and the generic excursion in \emph{every} vertex group of $\Gamma$ is logarithmic.
\end{corollary}

\end{document}